\documentclass[12pt,a4paper,reqno]{amsart} %optio titlepage!

\usepackage[T1]{fontenc}       % kirjaimet, joissa aksentteja (skandit)

\usepackage{amsfonts}          % AMS-fontit

\usepackage{amsmath}

\usepackage{amssymb}

\usepackage{overpic}
\usepackage{euscript} % \mathcal tuottaa hyvдnnдkцisen fontin
\usepackage{eurosym}
\usepackage{comment}
\usepackage{mathrsfs} % calligraphic font via the command \mathscr{ABC}
\usepackage{ifthen}
\usepackage{esint} %k.a.integraali komennolla \fint
\usepackage{color}

\long\def\symbolfootnote[#1]#2{\begingroup%
\def\thefootnote{\fnsymbol{footnote}}\footnote[#1]{#2}\endgroup}

{\qed\vspace{5pt}}

\usepackage{amsthm}

\newtheoremstyle{lause}% name
{5pt}% space above
{5pt}% space below
{\slshape}% body font
{\parindent}% indent amount (empty = no indent)
{\bfseries}% theorem head font
{.}% punctuation after theorem head
{.5em}% space after theorem head
{}% theorem head spec (can be left empty, meaning 'normal')

\theoremstyle{lause}

\newtheoremstyle{maaritelma}% name
{5pt}% space above
{5pt}% space below
{\rmfamily}% body font
{\parindent}% indent amount (empty = no indent)
{\bfseries}% theorem head font
{.}% punctuation after theorem head
{.5em}% space after theorem head
{}% theorem head spec (can be left empty, meaning 'normal')

\theoremstyle{maaritelma}
\newtheoremstyle{lause}% name
{5pt}% space above
{5pt}% space below
{\slshape}% body font
{\parindent}% indent amount (empty = no indent)
{\bfseries}% theorem head font
{.}% punctuation after theorem head
{.5em}% space after theorem head
{}% theorem head spec (can be left empty, meaning 'normal')

\theoremstyle{lause}
\newtheorem{theorem}{Theorem}[section]
\newtheorem{lemma}[theorem]{Lemma}

\newtheorem{corollary}[theorem]{Corollary}

\newtheoremstyle{maaritelma}% name
{5pt}% space above
{5pt}% space below
{\rmfamily}% body font
{\parindent}% indent amount (empty = no indent)
{\bfseries}% theorem head font
{.}% punctuation after theorem head
{.5em}% space after theorem head
{}% theorem head spec (can be left empty, meaning 'normal')

\theoremstyle{maaritelma}
\newtheorem{definition}[theorem]{Definition}

\newtheorem{example}[theorem]{Example}
\newtheorem{remark}[theorem]{Remark}

\newtheorem{open}[theorem]{Open question}

\usepackage{graphicx}
\usepackage{caption}
\numberwithin{equation}{section}

\begin{document}

\thispagestyle{empty}

\begin{center}

{\large{\textbf{Balayage, equilibrium measure, and Deny's principle of positivity of mass for $\alpha$-Green potentials}}}

\vspace{18pt}

\textbf{Natalia Zorii}

\vspace{18pt}

\emph{In memory of Lawrence Zalcman}\vspace{8pt}

\footnotesize{\address{Institute of Mathematics, Academy of Sciences
of Ukraine, Tereshchenkivska~3, 02000, Kyiv, Ukraine\\
natalia.zorii@gmail.com }}

\end{center}

\vspace{12pt}

{\footnotesize{\textbf{Abstract.} In the theory of $g_\alpha$-potentials on a domain $D\subset\mathbb R^n$, $n\geqslant2$, $g_\alpha$ being the $\alpha$-Green kernel associated with the $\alpha$-Riesz kernel $|x-y|^{\alpha-n}$ of order $\alpha\in(0,n)$, $\alpha\leqslant2$, we establish the existence and uniqueness of the $g_\alpha$-balayage $\mu^F$ of a positive Radon measure $\mu$ onto a relatively closed set $F\subset D$, we analyze its alternative characterizations, and we provide necessary and/or sufficient conditions for $\mu^F(D)=\mu(D)$ to hold, given
in terms of the $\alpha$-harmonic measure of suitable Borel subsets of $\overline{\mathbb R^n}$, the one-point compactification of $\mathbb R^n$. As a by-product, we find necessary and/or sufficient conditions for the existence of the $g_\alpha$-equ\-il\-ibrium measure $\gamma_F$, $\gamma_F$ being understood in an extended sense where $\gamma_F(D)$ might be infinite. We also discover quite a surprising version of Deny's principle of positivity of mass for $g_\alpha$-potentials, thereby significantly improving a previous result by Fuglede and Zorii (Ann.\ Acad.\ Sci.\ Fenn.\ Math., 2018). The results thus obtained are sharp, which is illustrated by means of a number of examples. Some open questions are also posed.
}}
\symbolfootnote[0]{\quad 2010 Mathematics Subject Classification: Primary 31C15.}
\symbolfootnote[0]{\quad Key words: $\alpha$-Green balayage; $\alpha$-Green equilibrium measure; $\alpha$-harmonic measure; $\alpha$-thinness of a set at infinity; a refined version of Deny's principle of positivity of mass for $\alpha$-Green potentials.
}

\vspace{6pt}

\markboth{\emph{Natalia Zorii}} {\emph{Balayage,  equilibrium measure, and Deny's principle of positivity of mass for $\alpha$-Green potentials}}

\section{General conventions and preliminaries}\label{sec1'} Fix an (open, connected) domain $D\subset\mathbb R^n$, $n\geqslant2$, which might coincide with the whole $\mathbb R^n$, and a positive (Radon) measure $\mu$ on $\mathbb R^n$ that is concentrated on $D$. (Being {\it concentrated} on a Borel set $B\subset\mathbb R^n$ means that $\mu(B^c)=0$, where $B^c:=\mathbb R^n\setminus B$, or equivalently that $\mu=\mu|_B$, where $\mu|_B$ is the {\it trace} of $\mu$ to $B$.) All such $\mu$ form the convex cone, denoted by $\mathfrak M^+(D)$. If $D=\mathbb R^n$, we simply write $\mathfrak M^+:=\mathfrak M^+(\mathbb R^n)$.

Any element of $\mathfrak M^+(D)$ can obviously be thought of as a positive (Radon) measure on $D$, where $D$ is now treated as a locally compact (Hausdorff) space. For the theory of measures on a locally compact space, we refer to Bourbaki \cite{B2} or Edwards \cite{E2}.

However, not every measure $\xi$ on the locally compact space $D$ can be extended to a measure on $\mathbb R^n$ (unless, of course,  $\xi$ is {\it bounded}, i.e.\ $\xi(D)<\infty$). Namely, $\mathfrak M^+(D)$ consists of all positive measures on the locally compact space $D$ that are {\it extendible} by $0$ outside $D$ to all of $\mathbb R^n$; or equivalently, $\mathfrak M^+(D)$ is the class of all $\nu|_D$, $\nu$ ranging over $\mathfrak M^+$. We use one and the same symbol for $\mu\in\mathfrak M^+(D)$ interpreted either as a measure on the locally compact space $D$, or as the trace $\nu|_D$ of some $\nu\in\mathfrak M^+$ to $D$. The extension of $\mu\in\mathfrak M^+(D)$ by $0$ outside $D$ to $\mathbb R^n$ will be denoted by $\mu$ as well.

Fix $\alpha\in(0,n)$, $\alpha\leqslant2$. It follows from the above that for $\mu\in\mathfrak M^+(D)$, one can apply the theory of potentials with respect to both the $\alpha$-Riesz and $\alpha$-Green kernels, denoted by $\kappa_\alpha$ and $g_\alpha$, respectively. That is, $\kappa_\alpha(x,y):=|x-y|^{\alpha-n}$, $x,y\in\mathbb R^n$, and
\begin{equation}\label{g}g_\alpha(x,y):=g^\alpha_D(x,y):=U_{\kappa_\alpha}^{\varepsilon_x}(y)-
U_{\kappa_\alpha}^{(\varepsilon_x)_{\kappa_\alpha}^Y}(y),\quad
(x,y)\in D\times\mathbb R^n,\end{equation}
where $U_{\kappa_\alpha}^\nu(\cdot):=\int\kappa_\alpha(\cdot,z)\,d\nu(z)$ is the {\it $\kappa_\alpha$-potential} of $\nu\in\mathfrak M^+$, $\varepsilon_x$ is the unit Dirac measure at $x\in D$, and $(\varepsilon_x)_{\kappa_\alpha}^Y$ is the {\it $\kappa_\alpha$-balayage} of $\varepsilon_x$ onto the closed (in $\mathbb R^n$) set $Y:=D^c$, uniquely characterized within $\mathfrak M^+(Y)$ by means of the equality
\begin{equation}\label{gg}U_{\kappa_\alpha}^{(\varepsilon_x)_{\kappa_\alpha}^Y}=U_{\kappa_\alpha}^{\varepsilon_x}\quad\text{quasi-everywhere on $Y$}.\end{equation}
(For more details, see e.g.\ \cite{Fr,FZ,L,R}. Recall that {\it quasi-everywhere} ({\it q.e.}) means that the equality holds on all of $Y$, except for a subset of {\it outer} $\kappa_\alpha$-capacity zero. For the theory of inner and outer $\kappa_\alpha$-capacities, we refer to Landkof \cite[Section~II.2]{L}. It should however be stressed here that all the sets, appearing in the current work, are {\it Borel}, so that their inner and outer $\kappa_\alpha$-cap\-acities coincide \cite[Theorem~2.8]{L}.)

When speaking of $\nu\in\mathfrak M^+$, we always understand that $U_{\kappa_\alpha}^\nu$ is not identically infinite on $\mathbb R^n$, which according to \cite[Section~I.3.7]{L} occurs if and only if
\begin{equation}\label{intf}
\int_{|y|>1}\frac{d\nu(y)}{|y|^{n-\alpha}}<\infty.
\end{equation}
Then, actually, $U_{\kappa_\alpha}^\nu$ is finite q.e.\ on $\mathbb R^n$ \cite[Section~III.1.1]{L}. Note that (\ref{intf}) necessarily holds if $\nu$ is either bounded, or of finite {\it $\kappa_\alpha$-energy} $I_{\kappa_\alpha}(\nu)$ (cf.\ (\ref{Eg}) with $D:=\mathbb R^n$).

This work is to continue the investigation of the theory of $g_\alpha$-pot\-entials, initiated by Riesz \cite{R} and Frostman \cite{Fr} and proceeded further by Fuglede and Zorii \cite{FZ}. As in \cite{Fr,FZ,R}, in the present paper we deal only with analytic aspects of this theory, leaving aside its probabilistic counterpart, cf.\ \cite{BH,Doob}.

Let $\mathcal E^+_{g_\alpha}$ stand for the cone of all positive measures $\xi$ on $D$ of finite {\it $g_\alpha$-energy}:
\begin{equation}\label{Eg}
I_{g_\alpha}(\xi):=\int_{D\times D} g_\alpha(x,y)\,d(\xi\otimes\xi)(x,y)<\infty.
\end{equation}
The main fact discovered in \cite{FZ} (see Theorem~4.11 therein) is that the kernel $g_\alpha$ is {\it perfect}, which means that
the cone $\mathcal E^+_{g_\alpha}$ is complete in the {\it strong} topology, introduced by means of the energy norm $\|\xi\|_{g_\alpha}:=\sqrt{I_{g_\alpha}(\xi)}$, and that the strong topology on $\mathcal E^+_{g_\alpha}$ is finer than the {\it vague} topology on $\mathcal E^+_{g_\alpha}$, the topology of pointwise convergence on the class $C_0(D)$ of all continuous functions $\varphi:D\to\mathbb R$ of compact (in $D$) support.\footnote{For $D=\mathbb R^n$, this goes back to Cartan \cite{Ca1} (if $\alpha=2$) and Deny \cite{D1} (if $\alpha\in(0,2]$). With regard to the concept of a perfect kernel on a locally compact space, see Fuglede \cite[Section~3.3]{F1}.}

Another useful fact discovered in \cite{FZ} (see Theorem~4.6 therein) is that the kernel $g_\alpha$ satisfies {\it the domination principle} in the following quite strong form: for any $\xi\in\mathcal E^+_{g_\alpha}$, any (extendible) $\mu\in\mathfrak M^+(D)$, and any positive $\alpha$-superharmonic function $\psi$ on $\mathbb R^n$ such that $U^\xi_{g_\alpha}\leqslant U^\mu_{g_\alpha}+\psi$ $\xi$-a.e., the same inequality holds true on all of $D$.\footnote{For the concept of {\it $\alpha$-superharmonic} function on $\mathbb R^n$, see \cite[Chapter~I]{L}.} Here $U_{g_\alpha}^\xi(\cdot):=\int g_\alpha(\cdot,y)\,d\xi(y)$ is the {\it $g_\alpha$-potential} of a measure $\xi$ on $D$.

The present work is concerned with the theory of {\it $g_\alpha$-balayage} and some relevant questions, see Sections~\ref{sec1}--\ref{sec-Deny} for the results obtained and Section~\ref{sec-proofs} for their proofs. The analysis performed is mainly based on the above-mentioned results from \cite{FZ}, the theory of $\kappa_\alpha$-balayage \cite{BH,FZ,L,Z-bal,Z-bal2}, and the theory of balayage on a locally compact space with respect to a perfect kernel satisfying suitable maximum principles, the latter being developed in the author's recent papers \cite{Z22}--\cite{Z-arx} and \cite{Z24b}--\cite{Z24}.

Let $c_{\kappa_\alpha}(B)$, resp.\ $c_{g_\alpha}(B)$, denote the {\it capacity} of a Borel set $B\subset\mathbb R^n$, resp.\ $B\subset D$, with respect to the kernel $\kappa_\alpha$, resp.\ $g_\alpha$. (For basic facts of the theory of $g_\alpha$-capacities, we refer to \cite[Section~4.4]{FZ}.) If $B\subset D$, then obviously
\begin{equation}\label{iff}
c_{\kappa_\alpha}(B)=0\iff c_{g_\alpha}(B)=0,
\end{equation}
$U_{\kappa_\alpha}^{(\varepsilon_x)_{\kappa_\alpha}^Y}(y)$, $x\in D$, being bounded when $y$ ranges over a compact subset of $D$.

A measure $\nu\in\mathfrak M^+$ is said to be {\it $c_{\kappa_\alpha}$-absolutely continuous} if for any Borel set $B\subset\mathbb R^n$ with $c_{\kappa_\alpha}(B)=0$, we have $\nu(B)=0$. This certainly occurs if $I_{\kappa_\alpha}(\nu)<\infty$, but not conversely \cite[pp.~134--135]{L}. If $\nu\in\mathfrak M^+(D)$, then, on account of (\ref{iff}), the concept of $c_{\kappa_\alpha}$-absolute continuity coincides with that of $c_{g_\alpha}$-absolute continuity, the latter concept being obtained from the former by replacing $c_{\kappa_\alpha}(\cdot)$ by $c_{g_\alpha}(\cdot)$.

For Borel $B\subset\mathbb R^n$, let $\breve{\mathfrak M}^+(B)$ stand for the subset of $\mathfrak M^+(B)$ consisting of all $c_{\kappa_\alpha}$-ab\-s\-ol\-utely continuous measures. In case $B=\mathbb R^n$, we simply write
$\breve{\mathfrak M}^+:=\breve{\mathfrak M}^+(\mathbb R^n)$.

For any closed $Q\subset\mathbb R^n$, let $\nu^Q_{\kappa_\alpha}$ denote the $\kappa_\alpha$-{\it balayage} of $\nu\in\mathfrak M^+$ onto $Q$, which can be defined e.g.\ as the only measure in $\mathfrak M^+$ satisfying the symmetry relation\footnote{For an alternative characterization of $\nu^Q_{\kappa_\alpha}$, see \cite[Lemma~3.10]{FZ}; compare with Theorem~\ref{th-char} below, dealing with the $g_\alpha$-balayage.}
\begin{equation}\label{sym}
I_{\kappa_\alpha}\bigl(\nu^Q_{\kappa_\alpha},\sigma\bigr)=I_{\kappa_\alpha}\bigl(\nu,\sigma^Q_{\kappa_\alpha}\bigr)
\quad\text{for all $\sigma\in\mathcal E^+_{\kappa_\alpha}$}.
\end{equation}
Here $I_{\kappa_\alpha}(\mu,\nu):=\int\kappa_\alpha(x,y)\,d(\mu\otimes\nu)(x,y)=\int U^\mu_{\kappa_\alpha}(x)\,d\nu(x)$ is the {\it mutual $\kappa_\alpha$-energy} of $\mu,\nu\in\mathfrak M^+$, $\mathcal E^+_{\kappa_\alpha}$ is the convex cone of all $\sigma\in\mathfrak M^+$ with $I_{\kappa_\alpha}(\sigma)<\infty$, and $\sigma^Q_{\kappa_\alpha}$ denotes the only measure in $\mathcal E^+_{\kappa_\alpha}(Q):=\mathcal E^+_{\kappa_\alpha}\cap\mathfrak M^+(Q)$ having the property
\[U_{\kappa_\alpha}^{\sigma^Q_{\kappa_\alpha}}=U_{\kappa_\alpha}^\sigma\quad\text{q.e.\ on $Q$}.\]
Such $\sigma^Q_{\kappa_\alpha}$ does exist, and it can be found as the {\it orthogonal projection} of $\sigma\in\mathcal E^+_{\kappa_\alpha}$ onto the convex cone $\mathcal E^+_{\kappa_\alpha}(Q)$ in the pre-Hilbert space $\mathcal E_{\kappa_\alpha}:=\mathcal E^+_{\kappa_\alpha}-\mathcal E^+_{\kappa_\alpha}$, equipped with the inner product $\langle\nu,\lambda\rangle_{\kappa_\alpha}:=I_{\kappa_\alpha}(\nu,\lambda)$. See \cite{FZ} (Theorem~3.1 and Definition~3.4); for the concept of orthogonal projection, see Edwards \cite[Sections~1.12.3, 1.12.4]{E2}.

\begin{remark}
To justify an application of \cite[Theorem~1.12.3]{E2}, we observe that the cone $\mathcal E^+_{\kappa_\alpha}(Q)$ is {\it strongly complete}. In fact, since the kernel $\kappa_\alpha$ is perfect, $\mathcal E^+_{\kappa_\alpha}$ is strongly complete, and the strong topology on $\mathcal E^+_{\kappa_\alpha}$ is finer than the vague topology. As $\mathcal E^+_{\kappa_\alpha}(Q)$ is vaguely closed \cite[Section~III.2, Proposition~6]{B2}, the claim follows.
\end{remark}

It is also useful to note that the symmetry relation (\ref{sym}) can be extended to arbitrary $\mu,\nu\in\mathfrak M^+$ \cite[Theorem~3.8]{FZ}, i.e.
\begin{equation}\label{sym'}
I_{\kappa_\alpha}\bigl(\mu^Q_{\kappa_\alpha},\nu\bigr)=I_{\kappa_\alpha}\bigl(\mu,\nu^Q_{\kappa_\alpha}\bigr)
\quad\text{for all $\mu,\nu\in\mathfrak M^+$}.
\end{equation}

\section{On the existence and uniqueness of the $g_\alpha$-balayage}\label{sec1}

Fix $F\subset D$, $F\ne D$, that is relatively closed in $D$, and let $F^r$ denote the set of all {\it $\kappa_\alpha$-reg\-ul\-ar} points $x\in F$. By the Wiener type criterion \cite[Theorem~5.2]{L},
\begin{equation*}x\in F^r\iff\sum_{j\in\mathbb N}\,\frac{c_{\kappa_\alpha}(F_j)}{q^{j(n-\alpha)}}=\infty,\end{equation*}
where $q\in(0,1)$ and $F_j:=F\cap\bigl\{y\in\mathbb R^n:\ q^{j+1}<|x-y|\leqslant q^j\bigr\}$.
Alternatively, $x\in F^r$ if and only if $(\varepsilon_x)^{F\cup Y}_{\kappa_\alpha}=\varepsilon_x$ \cite[Section~6]{Z-bal}; or, on account of (\ref{sym'}),
\begin{equation}\label{reg}x\in F^r\iff U^{\nu^{F\cup Y}_{\kappa_\alpha}}_{\kappa_\alpha}(x)=U^\nu_{\kappa_\alpha}(x)\quad\text{for all $\nu\in\mathfrak M^+$}.\end{equation}
The set $F^r$ is Borel measurable \cite[Theorem~5.2]{Z-bal2}, while $F^i:=F\setminus F^r$, the set of all {\it $\kappa_\alpha$-ir\-reg\-ul\-ar} points $x\in F$, is of $\kappa_\alpha$-capacity zero (the Kel\-logg--Ev\-ans type theorem, see \cite[Theorem~6.6]{Z-bal}).

\begin{theorem}\label{bal-ex}
 For any $\mu\in\mathfrak M^+(D)$ such that $\mu|_{F^r}$ is $c_{\kappa_\alpha}$-absolutely continuous,\footnote{If either of $I_{\kappa_\alpha}(\mu)$ or $I_{g_\alpha}(\mu)$ is finite, then the assumption $\mu|_{F^r}\in\breve{\mathfrak M}^+(F^r)$ is superfluous. (Note that any $\mu\in\mathfrak M^+(D)$ of finite $\kappa_\alpha$-energy is also of finite $g_\alpha$-energy, but not the other way around.)} i.e.\ $\mu|_{F^r}\in\breve{\mathfrak M}^+(F^r)$, there exists the only measure $\mu^F_{g_\alpha}\in\breve{\mathfrak M}^+(F)$ having the property
 \begin{equation}\label{bal1'}
 U_{g_\alpha}^{\mu^F_{g_\alpha}}(y)=U_{g_\alpha}^\mu(y)\quad\text{for all $y\in F^r$};
 \end{equation}
 this $\mu^F_{g_\alpha}$ is said to be the $g_\alpha$-balayage of $\mu$ onto $F$.
 Actually,\footnote{For $\mu:=\varepsilon_x$, where $x\in D\setminus F$, cf.\ also Frostman \cite[Section~5]{Fr}.}
 \begin{equation}\label{bal2}
 \mu^F_{g_\alpha}=\mu^{F\cup Y}_{\kappa_\alpha}\bigl|_F.
 \end{equation}
 If moreover $\mu\in\mathcal E^+_{g_\alpha}$, then the same $\mu^F_{g_\alpha}$ can alternatively be found as the only measure in the cone $\mathcal E^+_{g_\alpha}(F):=\mathcal E^+_{g_\alpha}\cap\mathfrak M^+(F)$ such that
 \begin{equation}\label{pr}
 \|\mu-\mu^F_{g_\alpha}\|_{g_\alpha}=\min_{\nu\in\mathcal E^+_{g_\alpha}(F)}\,\|\mu-\nu\|_{g_\alpha}.
 \end{equation}
 \end{theorem}

\begin{remark}\label{rem1} The statement on the uniqueness in Theorem~\ref{bal-ex} would fail to hold if we dropped our limitation to the $c_{\kappa_\alpha}$-absolutely continuous measures of class $\mathfrak M^+(F)$. Indeed, if $\mu:=\varepsilon_x$, where $x\in F^i$, then there exist, in fact, infinitely many $\nu\in\mathfrak M^+(F)$ meeting (\ref{bal1'}), for so is every  $\nu:=a\varepsilon_x+b(\varepsilon_x)^F_{g_\alpha}$ with $a,b\in[0,1]$ and $a+b=1$, $(\varepsilon_x)^F_{g_\alpha}$ being uniquely determined by means of Theorem~\ref{bal-ex} with $\mu:=\varepsilon_x$.
\end{remark}

$\blacktriangleright$ In what follows, unless explicitly indicated otherwise, $\mu$ and $\mu^F_{g_\alpha}$ are as stated in Theorem~\ref{bal-ex}. This convention will generally not be repeated henceforth.

\begin{corollary}\label{cor0}For any Borel set $e\subset D$,
\begin{equation}\label{MMM}
\mu^F_{g_\alpha}(e)=\int d\mu(x)\int1_e(y)\,d(\varepsilon_x)_{\kappa_\alpha}^{F\cup Y}(y),
\end{equation}
$1_e$ being the indicator function of $e$.
\end{corollary}

\begin{proof}By virtue of (\ref{bal2}),
\begin{equation}\label{MMM'}\mu^F_{g_\alpha}(e)=\int1_e\,d\mu^{F\cup Y}_{\kappa_\alpha}\bigl|_F=\int1_{e\cap F}\,d\mu^{F\cup Y}_{\kappa_\alpha}.\end{equation}
The function $x\mapsto(\varepsilon_x)_{\kappa_\alpha}^{F\cup Y}$, $x\in\mathbb R^n$, being {\it scalar essentially $\mu$-integrable} (see Bourbaki \cite[Section~V.3.1]{B2} for a definition, and Zorii \cite[p.~464]{Z-bal2} for a proof), it is {\it $\mu$-ad\-e\-q\-u\-ate} \cite[Section~V.3, Definition~1 and Proposition~2]{B2}. Therefore, the integral representation for $\mu^{F\cup Y}_{\kappa_\alpha}$ holds \cite[Theorem~5.1]{Z-bal2}, which together with (\ref{MMM'}) implies (\ref{MMM}), by use of \cite[Section~V.3, Proposition~5(b)]{B2}. (With this regard, see also Remark~\ref{rem3} and Open question~\ref{PR} below, related to the integral representation for $g_\alpha$-swept measures.)\end{proof}

\begin{corollary}\label{cor1} We have
\begin{equation}\label{est2'}
U^{\mu^F_{g_\alpha}}_{g_\alpha}\leqslant U^\mu_{g_\alpha}\quad\text{everywhere on $D$},
\end{equation}
whence
\begin{equation}\label{est2}
\mu^F_{g_\alpha}(D)\leqslant\mu(D).
\end{equation}
\end{corollary}

\begin{proof}
In fact, (\ref{est2'}) follows from (\ref{bal1'}) by applying the domination principle as stated in Theorem~\ref{th-dom}, whereas (\ref{est2}) is implied by (\ref{est2'}) by employing the principle of positivity of mass for $g_\alpha$-potentials \cite[Theorem~4.13]{FZ}. Alternatively, (\ref{est2}) is obtained from (\ref{MMM}) with $e:=D$ in view of the inequality $(\varepsilon_x)_{\kappa_\alpha}^{F\cup Y}(\mathbb R^n)\leqslant1$ \cite[Eq.~(3.18)]{FZ}.
\end{proof}

In Theorems~\ref{th-balM1}, \ref{th-balM2}, \ref{th-eq} and Corollary~\ref{cor-balM0}, we provide necessary and/or sufficient conditions for the inequality in (\ref{est2}) to be equality. These are given in terms of the $\alpha$-harmonic measure of suitable Borel subsets of $\overline{\mathbb R^n}$, the one-point compactification of $\mathbb R^n$, and they are shown to be sharp (Examples~\ref{ex1}--\ref{ex2}, \ref{ex4'}, \ref{ex4''}).

\begin{corollary} If $Q\subset F$ is relatively closed in $D$, then\footnote{As in Landkof \cite[p.~264]{L}, this might be referred to as the $g_\alpha$-balayage "{\it with a rest}".}
\begin{equation}\label{restg}
\mu_{g_\alpha}^Q=\bigl(\mu_{g_\alpha}^F\bigr)_{g_\alpha}^Q.
\end{equation}
\end{corollary}

\begin{proof}By Theorem~\ref{bal-ex}, both $\mu_{g_\alpha}^Q$ and $(\mu_{g_\alpha}^F)_{g_\alpha}^Q$ belong to  $\breve{\mathfrak M}^+(Q)$, and moreover
\[U^{(\mu_{g_\alpha}^F)_{g_\alpha}^Q}_{g_\alpha}=U^{\mu_{g_\alpha}^F}_{g_\alpha}=U^\mu_{g_\alpha}=U^{\mu_{g_\alpha}^Q}_{g_\alpha}
\quad\text{on $Q^r$}.\]
In view of the uniqueness of the $g_\alpha$-swept measure (Theorem~\ref{bal-ex}), (\ref{restg}) follows.
\end{proof}

\begin{corollary}\label{cor-rg}
If $c_{\kappa_\alpha}(Y)=0$, then $\mu^F_{g_\alpha}=\mu^{F\cup Y}_{\kappa_\alpha}$.
\end{corollary}

\begin{proof}
In fact, since $\mu^{F\cup Y}_{\kappa_\alpha}$ is $c_{\kappa_\alpha}$-absolutely continuous (see Theorem~\ref{th-balR} below), $c_{\kappa_\alpha}(Y)=0$ yields $\mu^{F\cup Y}_{\kappa_\alpha}=\mu^{F\cup Y}_{\kappa_\alpha}\bigl|_F$, which combined with (\ref{bal2}) proves the corollary.
\end{proof}

We end this section with Theorem~\ref{th-char}, providing an alternative characterization of $g_\alpha$-swept measures.\footnote{If either $I_{g_\alpha}(\mu)<\infty$, or $\mu$ is bounded and concentrated on $D\setminus F$, see also \cite{Z22}--\cite{Z-arx}, \cite{Z24b}--\cite{Z24} for a number of alternative characterizations of $\mu_{g_\alpha}^F$. Note that those papers are applicable here, for the kernel $g_\alpha$ is perfect and satisfies the complete maximum principle \cite[Theorems~4.6, 4.11]{FZ}.}
We denote
\begin{align}\label{gamma}
\Gamma_{g_\alpha}(\mu,F;D)&:=\bigl\{\theta\in\mathfrak M^+(D):\text{ \ $U_{g_\alpha}^\theta\geqslant U_{g_\alpha}^\mu$ \ q.e.\ on $F$}\bigr\},\\
\notag\breve{\Gamma}_{g_\alpha}(\mu,F;D)&:=\Gamma_{g_\alpha}(\mu,F;D)\cap\breve{\mathfrak M}^+(D).\end{align}

\begin{theorem}\label{th-char}The $g_\alpha$-balayage
$\mu_{g_\alpha}^F$ can equivalently be determined as the only measure in the class $\breve{\Gamma}_{g_\alpha}(\mu,F;D)$ having the property
\begin{equation}\label{eqchar}
U_{g_\alpha}^{\mu_{g_\alpha}^F}=\min_{\theta\in\breve{\Gamma}_{g_\alpha}(\mu,F;D)}\,U_{g_\alpha}^\theta=
\min_{\theta\in\Gamma_{g_\alpha}(\mu,F;D)}\,U_{g_\alpha}^\theta\quad\text{on all of $D$}.
\end{equation}
\end{theorem}

\begin{corollary}\label{cor-mass} It holds true that
\begin{equation}\label{minm}
\mu_{g_\alpha}^F(D)=\min_{\theta\in\breve{\Gamma}_{g_\alpha}(\mu,F;D)}\,\theta(D)=\min_{\theta\in\Gamma_{g_\alpha}(\mu,F;D)}\,\theta(D).
\end{equation}
\end{corollary}

\begin{proof}
  This follows at once from (\ref{eqchar}) by applying the principle of positivity of mass for $g_\alpha$-potentials \cite[Theorem~4.13]{FZ}.
\end{proof}

\begin{remark}
However, the extremal property (\ref{minm}) cannot serve as an equivalent definition of the $g_\alpha$-balayage, for it does not determine $\mu_{g_\alpha}^F$ uniquely. Indeed, let $D:=\mathbb R^n$, and let $F\ne\mathbb R^n$ be not $\alpha$-thin at infinity.\footnote{By Kurokawa and Mizuta \cite{KM}, a Borel set $Q\subset\mathbb R^n$ is said to be {\it $\alpha$-thin at infinity} if
\begin{equation*}
 \sum_{j\in\mathbb N}\,\frac{c_{\kappa_\alpha}(Q_j)}{q^{j(n-\alpha)}}<\infty,
 \end{equation*}
where $q\in(1,\infty)$ and $Q_j:=Q\cap\{x\in\mathbb R^n:\ q^j\leqslant|x|<q^{j+1}\}$. See also \cite[Section~2]{Z-bal2}.\label{f-KM}} Then for any $\mu\in\mathcal E^+_{\kappa_\alpha}(F^c)$,
\begin{equation}\label{22}\mu^F_{\kappa_\alpha}\ne\mu\quad\text{and}\quad\mu^F_{\kappa_\alpha}(\mathbb R^n)=\mu(\mathbb R^n),\end{equation}
the latter being obtained from \cite[Corollary~5.3]{Z-bal2} (cf.\ also \cite[Theorem~3.22]{FZ} or \cite[Theorem~4]{Z2}).
Since both $\mu$ and $\mu^F_{\kappa_\alpha}$ belong to $\breve{\Gamma}_{\kappa_\alpha}(\mu,F;\mathbb R^n)$, while this class is convex, combining (\ref{minm}) and (\ref{22}) shows that $\breve{\Gamma}_{\kappa_\alpha}(\mu,F;\mathbb R^n)$ has infinitely many measures meeting (\ref{minm}), for so is every $a\mu+b\mu^F_{\kappa_\alpha}$ with $a,b\in[0,1]$ and $a+b=1$.
\end{remark}

\section{How does $\mu(D)$ vary under the $g_\alpha$-balayage?}\label{how}

\subsection{The case $\alpha<2$} In this subsection, unless explicitly stated otherwise, $\alpha<2$.

\begin{theorem}\label{th-balM1}Under the requirements of Theorem~\ref{bal-ex}, if moreover
\begin{equation}\label{As}
\mu\ne\mu|_{F^r},\quad{\rm Int}_{\mathbb R^n}Y\ne\varnothing,
\end{equation}
then
\begin{equation}\label{Assa}
 \mu^F_{g_\alpha}(D)<\mu(D).
\end{equation}
\end{theorem}

Theorem~\ref{th-balM1} is sharp in the sense that, in general, it would fail to hold if any of its requirements were dropped (see  Examples~\ref{ex1}--\ref{ex2} below).

We denote by $B_{x,r}$ the open ball of radius $r$ centered at $x$, and $B_r:=B_{0,r}$.

\begin{example}\label{ex1}Let $\alpha=2$, $D:=B_1$, and $F:=D\setminus B_{1/2}$. For any $y\in B_{1/2}$, $(\varepsilon_y)^{F\cup Y}_{\kappa_2}(\mathbb R^n)=1$, $F\cup Y$ not being $2$-thin at infinity, and moreover $S\bigl((\varepsilon_y)^{F\cup Y}_{\kappa_2}\bigr)=\partial_DF$, see  \cite[Theorem~8.5]{Z-bal}.\footnote{Henceforth, $S(\nu)$ denotes the support of $\nu\in\mathfrak M^+$ in $\mathbb R^n$.} On account of (\ref{bal2}), this gives $(\varepsilon_y)^F_{g_2}(D)=1$, which shows that the assumption $\alpha<2$ is indeed important for the validity of Theorem~\ref{th-balM1}.\end{example}

\begin{example}\label{ex3} To confirm the importance of $\mu\ne\mu|_{F^r}$ for the validity of Theorem~\ref{th-balM1}, let $D:=B_1$, $F:=\{|x|\leqslant2^{-1}\}$, and let $\mu$ be the $g_\alpha$-equilibrium measure on $F=F^r$. Since $\mu\in\mathcal E^+_{g_\alpha}(F)$, the $g_\alpha$-balayage $\mu^F_{g_\alpha}$ is actually the orthogonal projection of $\mu$ onto $\mathcal E^+_{g_\alpha}(F)$, cf.\ (\ref{pr}), hence $\mu^F_{g_\alpha}=\mu$, and consequently $\mu^F_{g_\alpha}(D)=\mu(D)$.\end{example}

\begin{example}\label{ex2}Let $Y$ consist of finitely many points in $B_1$, $F:=\{|x|\geqslant1\}$, and let $\mu:=\varepsilon_y$, where $y\in B_1\setminus Y$. Then, by Corollary~\ref{cor-rg}, \[(\varepsilon_y)^F_{g_\alpha}=(\varepsilon_y)^{F\cup Y}_{\kappa_\alpha}=(\varepsilon_y)^F_{\kappa_\alpha},\]
$(\varepsilon_y)^{F\cup Y}_{\kappa_\alpha}$ being $c_{\kappa_\alpha}$-absolutely continuous (see Theorem~\ref{th-balR} below), whence
\[(\varepsilon_y)^F_{g_\alpha}(D)=(\varepsilon_y)^F_{\kappa_\alpha}(F)=
(\varepsilon_y)^F_{\kappa_\alpha}(\mathbb R^n)=1,\]
where the last equality holds true by \cite[Theorem~3.22]{FZ}, $F$ not being $\alpha$-thin at infinity. This illustrates the importance of ${\rm Int}_{\mathbb R^n}Y\ne\varnothing$ for the validity of Theorem~\ref{th-balM1}.
\end{example}

\begin{remark}\label{rem3}
Assume that $\alpha\in(1,2)$, $D$ is a bounded domain of class $C^{1,1}$,\footnote{A domain $D$ is said to be {\it of class} $C^{1,1}$ if for every $y\in\partial D:=\partial_{\mathbb R^n}D$, there exist $B_{x,r}\subset D$ and $B_{x',r}\subset D^c$, where $r>0$, that are tangent at $y$, see \cite[p.~458]{Bogdan}.\label{F11}} and $\mu$ is a positive measure concentrated on $D\setminus F$ and such that $\mu(D)$ or $I_{g_\alpha}(\mu)$ is finite. In this quite a particular case, Theorem~\ref{th-balM1} was recently verified by the author \cite[Theorem~4.2]{Z24} with the aid of the integral representation
\begin{equation}\label{int}
\mu^F_{g_\alpha}=\int(\varepsilon_x)^F_{g_\alpha}\,d\mu(x),
\end{equation}
established in \cite{Z24} (see Theorem~1.5 and Remark~1.6 therein). It was crucial to the proof of (\ref{int}) that, for the above $\alpha$ and $D$, the set of all $\varphi\in C_0(D)$ representable as $g_\alpha$-potentials of signed measures of finite $g_\alpha$-energy is dense in the space $C_0(D)$ equipped with the inductive limit topology; this observation, in turn, was essentially based on Eq.~(19) by Bogdan and Jakubowski \cite{Bogdan} (see \cite[Remark~1.6]{Z24} for details).\footnote{As to {\it the inductive limit topology} on the space $C_0(X)$, $X$ being a second-countable, locally compact space, see Bourbaki \cite[Section~II.4.4]{B4} and \cite[Section~III.1.1]{B2}, cf.\ also \cite[Section~4.1]{Z-arx}.}
\end{remark}

\begin{open}\label{PR}
Under the assumptions of Theorem~\ref{bal-ex}, much more general than those in Remark~\ref{rem3}, does the integral representation (\ref{int}) still hold?\footnote{For $\alpha=2$, see Doob \cite[Section~1.X, Eq.~(5.2)]{Doob}. In the case where $c_{\kappa_\alpha}(Y)=0$, (\ref{int}) can be derived from Bliedtner and Hansen \cite[p.~255]{BH} and Zorii \cite[Theorem~5.1]{Z-bal2}.}
\end{open}

\subsection{The case $0<\alpha\leqslant2$}\label{sec-harm} A set $e\subset\overline{\mathbb R^n}$ is said to be {\it Borel} if so is $e\cap\mathbb R^n$. Here $\overline{\mathbb R^n}:=\mathbb R^n\cup\{\infty_{\mathbb R^n}\}$, $\infty_{\mathbb R^n}$ being the Alexandroff point of $\mathbb R^n$.

Define $\Omega:=D\setminus F$. For any $x\in\Omega$ and any Borel $e\subset\overline{\mathbb R^n}$, we introduce the (fractional) {\it $\alpha$-harmonic measure} $\omega_\alpha(x,e;\Omega)$ by setting
\begin{equation}\label{def-h}
\omega_\alpha(x,e;\Omega)=\left\{
\begin{array}{cl}(\varepsilon_x)^{\Omega^c}_{\kappa_\alpha}(e)&\text{\ if $e\subset\mathbb R^n$},\\
(\varepsilon_x)^{\Omega^c}_{\kappa_\alpha}(e\cap\mathbb R^n)+\omega_\alpha(x,\{\infty_{\mathbb R^n}\};\Omega)&\text{\ otherwise},\\ \end{array} \right.\end{equation}
where
\begin{equation}\label{Def}
\omega_\alpha(x,\{\infty_{\mathbb R^n}\};\Omega):=1-(\varepsilon_x)^{\Omega^c}_{\kappa_\alpha}(\mathbb R^n).
\end{equation}
Thus,
\[\omega_\alpha(x,\overline{\mathbb R^n};\Omega)=1\quad\text{for all $x\in\Omega$}.\]
Note that for $\alpha=2$, $S\bigl((\varepsilon_x)^{\Omega^c}_{\kappa_2}\bigr)\subset\partial\Omega$ for all $x\in\Omega$,
cf.\ \cite[Theorem~8.5]{Z-bal}, and hence the concept of $\alpha$-harmonic measure, introduced by (\ref{def-h}) and (\ref{Def}), generalizes that of $2$-har\-mo\-nic measure, defined in \cite[Section~IV.3.12]{L} for Borel subsets of $\partial\Omega$.

$\blacktriangleright$ In the rest of this section as well as in Sections~\ref{sec-exi} and \ref{sec-Deny}, we assume for simplicity of formulations that either $\alpha<2$ or the open set $\Omega$ is {\it connected}.

\begin{theorem}\label{th-balM2} The following {\rm(i)}--{\rm(iii)} are equivalent.
\begin{itemize}
\item[{\rm(i)}] $\mu^F_{g_\alpha}(D)=\mu(D)$ for any $\mu\in\mathfrak M^+(D)$ such that $\mu|_F\in\breve{\mathfrak M}^+$ and $\mu\ne\mu|_F$.
\item[{\rm(ii)}] $\omega_\alpha(x,\overline{\mathbb R^n}\setminus D;\Omega)=0$ for all $x\in\Omega$.
\item[{\rm(iii)}] $\Omega^c$ is not $\alpha$-thin at infinity, and
\begin{equation}\label{dc}
\omega_\alpha(x,Y;\Omega)=0\quad\text{for all $x\in\Omega$}.
\end{equation}
\end{itemize}
\end{theorem}

\begin{remark}
For $\alpha=2$, (\ref{dc}) can be rewritten in the apparently weaker form
\[\omega_2(x,\partial\Omega\cap\partial D;\Omega)=0\quad\text{for all $x\in\Omega$}.\]
Indeed, this is clear from $S\bigl((\varepsilon_x)^{\Omega^c}_{\kappa_2}\bigr)\subset\partial\Omega$, cf.\ \cite[Theorem~8.5]{Z-bal}.
\end{remark}

\begin{corollary}\label{cor-balM0} If $c_{\kappa_\alpha}(Y)=0$, then {\rm(i)}--{\rm(vi)} are equivalent, where:
\begin{itemize}
\item[{\rm(iv)}] $\omega_\alpha(x,\{\infty_{\mathbb R^n}\};\Omega)=0$ for all $x\in\Omega$.
\item[{\rm(v)}] $\omega_\alpha(x,\{\infty_{\mathbb R^n}\};\Omega)=0$ for some $x\in\Omega$.
\item[{\rm(vi)}] $\Omega^c$ {\rm(}equivalently, $F${\rm)} is not $\alpha$-thin at infinity.
\end{itemize}
\end{corollary}

\begin{proof}
As $c_{\kappa_\alpha}(Y)=0$, we have $(\varepsilon_x)^{\Omega^c}_{\kappa_\alpha}(Y)=0$ for all $x\in\Omega$, $(\varepsilon_x)^{\Omega^c}_{\kappa_\alpha}$ being $c_{\kappa_\alpha}$-ab\-s\-ol\-ut\-ely continuous (Theorem~\ref{th-balR}), and (\ref{def-h}) shows that (ii) and (iv) are indeed equivalent. In view of (\ref{Def}), the equivalence of (iv)--(vi) follows from \cite[Theorem~3.22]{FZ}.
\end{proof}

\begin{example}\label{ex4'}
Theorem~\ref{th-balM2} is sharp in the sense that it would fail if $\mu\ne\mu|_F$ were dropped from the hypotheses. Indeed, let $\alpha\in(0,2)$, $D:=\{|x|>1\}$, and let $F:={\rm Cl}_{\mathbb R^n}B_{x_0,1}$, where $x_0:=(3,0,\ldots,0)$. Then for any $\mu\in\mathcal E^+_{g_\alpha}(F)$, $\mu^F_{g_\alpha}(D)=\mu(D)$ in consequence of (\ref{pr}), in spite of the fact that (ii) fails to hold. To verify the latter, we note from \cite[Theorem~8.5]{Z-bal} that
\[S\bigl((\varepsilon_x)^{\Omega^c}_{\kappa_\alpha}\bigr)=\Omega^c=F\cup Y\quad\text{for every $x\in\Omega$},\]  whence \[\omega_\alpha(x,\overline{\mathbb R^n}\setminus D;\Omega)\geqslant(\varepsilon_x)^{\Omega^c}_{\kappa_\alpha}(Y)>0.\]
\end{example}

\begin{example}\label{ex4''}
For $\alpha=2$, the connectedness of $\Omega$ is important for the validity of Corollary~\ref{cor-balM0}. To confirm this, let $D:=\mathbb R^n$ and $F:=\{1\leqslant|y|\leqslant2\}$. Then $\Omega^c=F$ is $\alpha$-thin at infinity, and so (vi) fails, in spite of the fact that $\omega_2(x,\{\infty_{\mathbb R^n}\};\Omega)=0$ for all $x\in B_1$. To verify the latter, we note that, for all $x\in B_1$, $S\bigl((\varepsilon_x)^F_{\kappa_2}\bigr)=\partial B_1$ by \cite[Theorem~8.5]{Z-bal},
hence $(\varepsilon_x)^F_{\kappa_2}=(\varepsilon_x)^{B_1^c}_{\kappa_2}$,
and therefore $(\varepsilon_x)^{\Omega^c}_{\kappa_2}(\mathbb R^n)=(\varepsilon_x)^{B_1^c}_{\kappa_2}(\mathbb R^n)=1$.
\end{example}

\section{On the existence of the $g_\alpha$-equilibrium measure $\gamma_F$}\label{sec-exi}

\begin{definition}\label{def-eq}
Given a relatively closed $F\subset D$, $\gamma_F:=\gamma_{F,g_\alpha}\in\breve{\mathfrak M}^+(F)$ is said to be the {\it $g_\alpha$-equilibrium measure} of $F$ if
\begin{equation}\label{EQ}
U^{\gamma_F}_{g_\alpha}=1\quad\text{q.e.\ on $F$}.
\end{equation}
\end{definition}

\begin{lemma}\label{eq-un}
The $g_\alpha$-equilibrium measure $\gamma_F$ is unique (if it exists).
\end{lemma}

\begin{proof}
If (\ref{EQ}) also holds for some $\xi\in\breve{\mathfrak M}^+(F)$ in place of $\gamma_F$, then,
in view of the $c_{\kappa_\alpha}$-absolute continuity of $\xi+\gamma_F$, $U_{g_\alpha}^\xi=U_{g_\alpha}^{\gamma_F}$ $(\xi+\gamma_F)$-a.e., hence $I_{g_\alpha}(\xi-\gamma_F)=0$, and so $\xi=\gamma_F$, by the strict positive definiteness of the kernel $g_\alpha$ \cite[Theorem~4.9]{FZ}.
\end{proof}

If $c_{g_\alpha}(F)<\infty$, then such $\gamma_F$ does exist,  and moreover (see \cite[Theorem~4.12]{FZ})
\[I_{g_\alpha}(\gamma_F)=\gamma_F(D)=c_{g_\alpha}(F).\]
The analysis performed in \cite{FZ} used substantially the perfectness of the kernel $g_\alpha$, discovered in Theorem~4.11 therein (see Section~\ref{sec1'} above for details). However, in the case $c_{g_\alpha}(F)\leqslant\infty$, which we suppose to take place, this tool becomes insufficient.

\begin{theorem}\label{th-eq}
If $\gamma_F$ exists, then necessarily
\begin{equation}\label{eq-eq}
\mu^F_{g_\alpha}(D)<\mu(D)\quad\text{for all nonzero $\mu\in\mathfrak M^+(\Omega)$},
\end{equation}
which implies, in turn, that
\begin{equation}\label{eq-eq1}
\omega_\alpha(x_0,\overline{\mathbb R^n}\setminus D;\Omega)>0\quad\text{for some $x_0\in\Omega$}.
\end{equation}
If $c_{\kappa_\alpha}(Y)=0$, then these two implications can actually be reversed; and so\footnote{Assertion (\ref{eq-eq2}) can be rewritten in the following apparently stronger form: for $\gamma_F$ to exist, it is necessary and sufficient that  $\omega_\alpha(x,\{\infty_{\mathbb R^n}\};\Omega)>0$ for all $x\in\Omega$.}
\begin{equation}\label{eq-eq2}\text{$\gamma_F$ exists}\iff\text{$\omega_\alpha(x_0,\{\infty_{\mathbb R^n}\};\Omega)>0$ \ for some $x_0\in\Omega$},\end{equation}
or equivalently,\footnote{Compare with Landkof \cite[Theorem~5.1]{L} and Zorii \cite[Theorem~2.1]{Z-bal2}, dealing with $D:=\mathbb R^n$.}
\begin{equation}\label{eq-eq3}\text{$\gamma_F$ exists}\iff\text{$\Omega^c$ is $\alpha$-thin at infinity}.\end{equation}
\end{theorem}

\begin{figure}[htbp]
\begin{center}
\vspace{-.2in}
\hspace{-.1in}\includegraphics[width=4.6in]{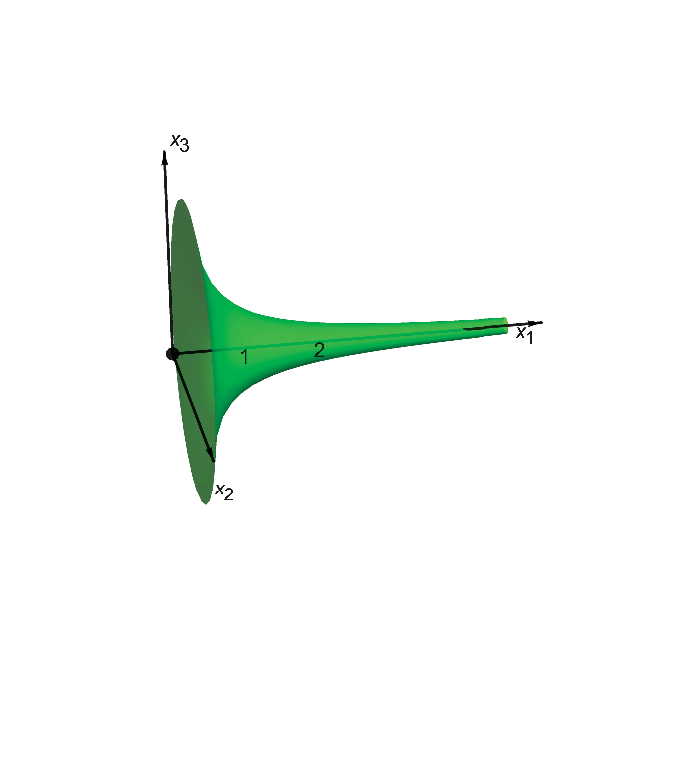}
\vspace{-1.8in}
\caption{The set $F_1$ in Example~\ref{ex} with $\varrho_1(x_1)=1/x_1$.\vspace{-.0in}}
\label{Fig2}
\end{center}
\end{figure}

\begin{example}\label{ex} Let $\alpha=2$ and $D:=\mathbb R^3\setminus Y$, where $Y$ consists of finitely many points in $\{(x_1,x_2,x_3)\in\mathbb R^3: x_1<0\}$. In the domain $D$, consider the rotation bodies
\begin{equation*}F_i:=\bigl\{x\in\mathbb R^3: \ 0\leqslant x_1<\infty, \
x_2^2+x_3^2\leqslant\varrho_i^2(x_1)\bigr\}, \ i=1,2,3,\end{equation*}
where
\begin{align*}
\varrho_1(x_1)&:=x_1^{-s}\text{ \ with\ }s\in[0,\infty),\\
\varrho_2(x_1)&:=\exp(-x_1^s)\text{ \ with\ }s\in(0,1],\\
\varrho_3(x_1)&:=\exp(-x_1^s)\text{ \ with\ }s\in(1,\infty).
\end{align*}
Then $F_1$ is not $2$-thin at infinity (Figure~\ref{Fig2}), $F_2$ is $2$-thin at infinity despite the fact that $c_{g_2}(F_2)=c_{\kappa_2}(F_2)=\infty$ (Figure~\ref{Fig1}),
whereas $F_3$ is $2$-thin at infinity, and moreover $c_{g_2}(F_3)=c_{\kappa_2}(F_3)<\infty$. Thus, according to the latter part of Theorem~\ref{th-eq}, $\gamma_{F_2,g_2}$ and $\gamma_{F_3,g_2}$ do exist (although $\gamma_{F_2,g_2}(F_2)=\infty$), whereas $\gamma_{F_1,g_2}$ fails to exist.
\end{example}

\begin{figure}[htbp]
\begin{center}
\vspace{-.2in}
\hspace{-1.1in}\includegraphics[width=4.0in]{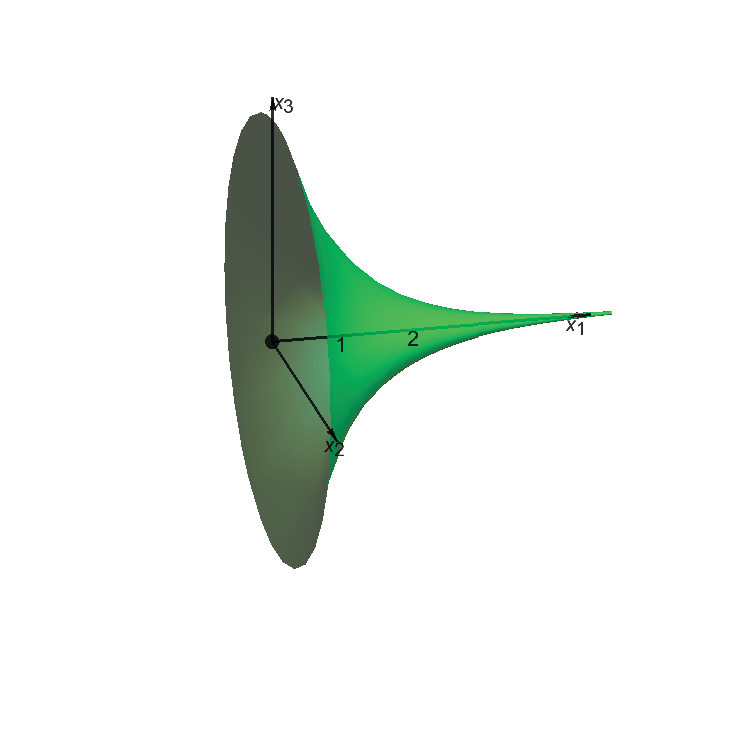}
\vspace{-1.0in}
\caption{The set $F_2$ in Example~\ref{ex} with $\varrho_2(x_1)=\exp(-x_1)$.\vspace{-.1in}}
\label{Fig1}
\end{center}
\end{figure}

\begin{example}\label{Ex}If $\alpha=2$, $D:=B_2$, and $\Omega:=B_{x,1}$, where $x:=(1,0,\ldots,0)$, then
\begin{equation}\label{eq-ex}
\omega_2(x,\overline{\mathbb R^n}\setminus D;\Omega)=0\quad\text{for all $x\in\Omega$},
\end{equation}
and so, by Theorem~\ref{th-eq}, $\gamma_{F,g_2}$ fails to exist, where $F:=D\setminus\Omega$.\end{example}

\begin{open}
If $c_{\kappa_\alpha}(Y)>0$, it is of interest to find sufficient conditions for the existence of the $g_\alpha$-equilibrium measure $\gamma_{F,g_\alpha}$.
\end{open}

We end this section by observing that, if $\gamma_F=\gamma_{F,g_\alpha}$ exists, then so does $\gamma_Q=\gamma_{Q,g_\alpha}$ for any relatively closed $Q\subset F$, and moreover
\begin{equation}\label{EQ'}\gamma_Q=(\gamma_F)^Q_{g_\alpha}.\end{equation}
Indeed, by Theorem~\ref{bal-ex}, $(\gamma_F)^Q_{g_\alpha}\in\breve{\mathfrak M}^+(Q)$ and
$U^{(\gamma_F)^Q_{g_\alpha}}_{g_\alpha}=U^{\gamma_F}_{g_\alpha}=1$ q.e.\ on $Q$, whence (\ref{EQ'}) in consequence of Lemma~\ref{eq-un}.

\section{A refined version of Deny's principle of positivity of mass}\label{sec-Deny}

By Fuglede and Zorii \cite[Theorem~4.13]{FZ}, for any $\mu,\nu\in\mathfrak M^+(D)$ such that
\[U_{g_\alpha}^\mu\leqslant U_{g_\alpha}^\nu\quad\text{on $D$},\] it holds true that $\mu(D)\leqslant\nu(D)$. (If $D=\mathbb R^n$ and $\alpha=2$, this goes back to Deny \cite{D2}.)

Under some additional requirements on the measures in question, the quoted result from \cite{FZ} can be refined as follows.

\begin{theorem}\label{th-D} Let $F\subset D$ be a relatively closed set having the property
\begin{equation}\label{zeroo}
\omega_\alpha(x,\overline{\mathbb R^n}\setminus D;\Omega)=0\quad\text{for all $x\in\Omega$}.
\end{equation}
Then for any $\mu,\nu\in\mathfrak M^+(D)$ such that $\mu|_F\in\breve{\mathfrak M}^+(F)$ and\footnote{If $D=\mathbb R^n$, then the requirement $\mu|_F\in\breve{\mathfrak M}^+(F)$ as well as that of the closedness of $F$ is superfluous, see the author's recent result \cite[Theorem~1.2]{Z-Deny}.}
\begin{equation}\label{insigma}
 U_{g_\alpha}^\mu\leqslant U_{g_\alpha}^\nu\quad\text{q.e.\ on $F$},
\end{equation}
we have
\begin{equation*}
\mu(D)\leqslant\nu(D).
\end{equation*}
\end{theorem}

Assumption (\ref{zeroo}) is important for the validity of Theorem~\ref{th-D}, which is seen from the following Theorem~\ref{th-D'}.

\begin{theorem}\label{th-D'} If {\rm(\ref{zeroo})} fails to hold, then there exist $\mu_0,\nu_0\in\mathfrak M^+(\Omega)$ such that  $U_{g_\alpha}^{\mu_0}\leqslant U_{g_\alpha}^{\nu_0}$ q.e.\ on $F$, but nevertheless, $\mu_0(D)>\nu_0(D)$.\end{theorem}

\begin{example}Let $\alpha=2$, $D:=B_2$, and $\Omega:=B_{x,1}$, where $x:=(1,0,\ldots,0)$. Then $\mu(D)\leqslant\nu(D)$ for any $\mu,\nu\in\breve{\mathfrak M}^+(D)$ such that $U_{g_2}^\mu\leqslant U_{g_2}^\nu$ q.e.\ on $F:=D\setminus\Omega$. Indeed, on account of (\ref{eq-ex}), we deduce this from Theorem~\ref{th-D}.\end{example}

\begin{example}
In the assumptions and notations of Example~\ref{ex}, we note from Corollary~\ref{cor-balM0} that
$\omega_2(x,\overline{\mathbb R^3}\setminus D;D\setminus F_1)=0$ for all $x\in D\setminus F_1$,
$F_1$ not being $2$-thin at infinity. Thus, according to Theorem~\ref{th-D}, $\mu(D)\leqslant\nu(D)$ for any $\mu,\nu\in\mathfrak M^+(D\setminus F_1)$ such that $U_{g_2}^\mu\leqslant U_{g_2}^\nu$ q.e.\ on $F_1$. However, according to Theorem~\ref{th-D'}, this conclusion fails to hold for any of $F_2$ or $F_3$, those sets being $2$-thin at infinity.
\end{example}

\section{Proofs}\label{sec-proofs}

\subsection{Preliminaries} Fix a closed set $Q\subset\mathbb R^n$ and a measure $\nu\in\mathfrak M^+$ such that $\nu|_{Q^r}$ is $c_{\kappa_\alpha}$-ab\-solutely continuous, i.e.\ $\nu|_{Q^r}\in\breve{\mathfrak M}^+(Q^r)$. Then the concept of $\kappa_\alpha$-balayage $\nu^Q_{\kappa_\alpha}$, introduced by means of the symmetry relation (\ref{sym}), can be specified as follows.

\begin{theorem}\label{th-balR}For these $Q$ and $\nu$, the $\kappa_\alpha$-balayage $\nu^Q_{\kappa_\alpha}$ is the only measure in the class $\breve{\mathfrak M}^+(Q)$ having the property $U_{\kappa_\alpha}^{\nu^Q_{\kappa_\alpha}}=U_{\kappa_\alpha}^\nu$ q.e.\ on $Q$, or equivalently
\begin{equation}\label{balR}
U_{\kappa_\alpha}^{\nu^Q_{\kappa_\alpha}}(x)=U_{\kappa_\alpha}^\nu(x)\quad\text{for all $x\in Q^r$}.
\end{equation}
\end{theorem}

\begin{proof} Relation (\ref{balR}) is clear from (\ref{reg}), while $\nu^Q_{\kappa_\alpha}\in\breve{\mathfrak M}^+(Q)$ is implied by $\nu|_{Q^r}\in\breve{\mathfrak M}^+$, see \cite[Corollary~5.2]{Z-bal2}. To complete the proof, it remains to utilize the uniqueness theorem for $\kappa_\alpha$-potentials in the form presented in \cite[p.~178, Remark]{L}.\end{proof}

\begin{remark}\label{rem2} Theorem~\ref{th-balR} fails to hold if we drop our limitation to the $c_{\kappa_\alpha}$-ab\-sol\-utely continuous measures, see Remark~\ref{rem1} with $D:=\mathbb R^n$ and $F:=Q$.\end{remark}

\begin{lemma}\label{l-FZ}
For any $\zeta\in\mathfrak M^+(D)$, the $g_\alpha$-potential $U_{g_\alpha}^\zeta$ is finite q.e.\ on $\mathbb R^n$, and given by
\[U_{g_\alpha}^\zeta(x)=U_{\kappa_\alpha}^\zeta(x)-U_{\kappa_\alpha}^{\zeta^Y_{\kappa_\alpha}}(x),\quad x\in\mathbb R^n.\]
\end{lemma}

\begin{proof}
 See \cite[Lemma~4.4]{FZ}.
\end{proof}

\begin{theorem}\label{th-dom}
For any $\zeta,\theta\in\mathfrak M^+(D)$ such that $\zeta$ is $c_{\kappa_\alpha}$-absolutely continuous and
\[U_{g_\alpha}^\zeta\leqslant U_{g_\alpha}^\theta\quad\text{$\zeta$-a.e.},\]
the same inequality holds true on all of $D$.
\end{theorem}

\begin{proof}
For $\alpha=2$, see Armitage and Gardiner \cite[Theorem~5.1.11]{AG} and Doob \cite[Theorem~1.V.10]{Doob}. For $\alpha<2$, this follows by a slight modification of Fuglede and Zorii \cite[Proof of Theorem~4.6]{FZ} (for $D=\mathbb R^n$, see Landkof \cite[Theorem~1.29]{L}).\footnote{See also Bogdan, Jarohs, and Kania \cite[Lemma~1.34]{Bogdan2} and Bogdan and Hansen \cite[Lemma~3.5(ii)]{BHa} for similar assertions in the framework of probabilistic potential theory and balayage spaces.}
\end{proof}

\subsection{Proof of Theorem~\ref{bal-ex}} Fix a set $F\subset D$, $F\ne D$, that is relatively closed in $D$, and a measure $\mu\in\mathfrak M^+(D)$ such that $\mu|_{F^r}\in\breve{\mathfrak M}^+$. Then, according to Theorem~\ref{th-balR},
$\mu^{F\cup Y}_{\kappa_\alpha}$ is uniquely characterized within the class $\breve{\mathfrak M}^+(F\cup Y)$ by means of the equality
\begin{equation}\label{pr1}
U_{\kappa_\alpha}^{\mu^{F\cup Y}_{\kappa_\alpha}}(x)=U_{\kappa_\alpha}^\mu(x)\quad\text{for all $x\in(F\cup Y)^r$}.
\end{equation}

We claim that for any given $x\in F^r$,
\begin{equation}\label{pr2}
 U_{g_\alpha}^{\mu^{F\cup Y}_{\kappa_\alpha}\bigl|_F}(x)=U_{g_\alpha}^\mu(x).
\end{equation}
In fact, by use of (\ref{g}),
\begin{align*}
U_{g_\alpha}^{\mu^{F\cup Y}_{\kappa_\alpha}\bigl|_F}(x)&=\int_Fg_\alpha(x,z)\,d\mu^{F\cup Y}_{\kappa_\alpha}\bigl|_F(z)=
\int_{F\cup Y}g_\alpha(x,z)\,d\mu^{F\cup Y}_{\kappa_\alpha}(z)\\
{}&=\int_{F\cup Y}\kappa_\alpha(x,z)\,d\mu^{F\cup Y}_{\kappa_\alpha}(z)-
\int_{F\cup Y}U_{\kappa_\alpha}^{(\varepsilon_x)_{\kappa_\alpha}^Y}(z)\,d\mu^{F\cup Y}_{\kappa_\alpha}(z),
\end{align*}
the second equality being valid since $g_\alpha(x,z)=0$ q.e.\ on $Y$, cf.\ (\ref{gg}), while $\mu^{F\cup Y}_{\kappa_\alpha}$ is $c_{\kappa_\alpha}$-ab\-sol\-utely continuous.
This yields, by (\ref{pr1}) and the Lebesgue--Fubini theorem,
\begin{align}
U_{g_\alpha}^{\mu^{F\cup Y}_{\kappa_\alpha}\bigl|_F}(x)&=U^\mu_{\kappa_\alpha}(x)-
\int_{F\cup Y}d\mu^{F\cup Y}_{\kappa_\alpha}(z)\int_Y|z-y|^{\alpha-n}\,d(\varepsilon_x)_{\kappa_\alpha}^Y(y)\notag\\
{}&=U^\mu_{\kappa_\alpha}(x)-\int_Yd(\varepsilon_x)_{\kappa_\alpha}^Y(y)\int_{F\cup Y}|z-y|^{\alpha-n}\,d\mu^{F\cup Y}_{\kappa_\alpha}(z)\notag\\
{}&=U^\mu_{\kappa_\alpha}(x)-\int_YU^{\mu^{F\cup Y}_{\kappa_\alpha}}_{\kappa_\alpha}(y)\,d(\varepsilon_x)_{\kappa_\alpha}^Y(y)\notag\\
{}&=U^\mu_{\kappa_\alpha}(x)-\int_YU^\mu_{\kappa_\alpha}(y)\,d(\varepsilon_x)_{\kappa_\alpha}^Y(y)\notag\\
{}&=U^\mu_{\kappa_\alpha}(x)-\int_YU^{\mu_{\kappa_\alpha}^Y}_{\kappa_\alpha}(y)\,d\varepsilon_x(y)\notag\\
{}&=U^\mu_{\kappa_\alpha}(x)-U^{\mu_{\kappa_\alpha}^Y}_{\kappa_\alpha}(x),\label{pr3}\end{align}
whence (\ref{pr2}), cf.\ Lemma~\ref{l-FZ}. Note that the fourth equality in (\ref{pr3}) holds by virtue of the $c_{\kappa_\alpha}$-absolute continuity of the measure $(\varepsilon_x)_{\kappa_\alpha}^Y$ (see Theorem~\ref{th-balR}), whereas the fifth equality follows by applying the symmetry relation (\ref{sym'}).

It has thus been verified that $\mu^F_{g_\alpha}$, introduced by means of (\ref{bal2}), satisfies (\ref{bal1'}). We also observe from (\ref{bal2}) that $\mu^F_{g_\alpha}$ is $c_{\kappa_\alpha}$-absolutely continuous, because so is $\mu^{F\cup Y}_{\kappa_\alpha}$ (Theorem~\ref{th-balR}); hence, $\mu^F_{g_\alpha}\in\breve{\mathfrak M}^+(F)$. If now $\nu$ is another measure in $\breve{\mathfrak M}^+(F)$ having the property $U_{g_\alpha}^\nu(y)=U_{g_\alpha}^\mu(y)$ q.e.\ on $F$, then, by use of the countable subadditivity of outer $\kappa_\alpha$-capacity \cite[Eq.~(2.2.3)]{L}, $U_{g_\alpha}^{\nu-\mu^F_{g_\alpha}}(y)=0$ q.e.\ on $F$, whence
\begin{equation}\label{e}
 I_{g_\alpha}\bigl(\nu-\mu^F_{g_\alpha}\bigr)=0,
\end{equation}
the measure $\nu+\mu^F_{g_\alpha}\in\mathfrak M^+(F)$ being likewise $c_{\kappa_\alpha}$-absolutely continuous. Since the kernel $g_\alpha$ is strictly positive definite \cite[Theorem~4.9]{FZ}, (\ref{e}) yields $\nu=\mu^F_{g_\alpha}$, thereby completing the proof of the former part of the theorem.

It remains to consider the case of $\mu\in\mathfrak M^+(D)$ having finite $g_\alpha$-energy. As the kernel $g_\alpha$ is perfect and satisfies the domination principle \cite[Theorems~4.6, 4.11]{FZ}, applying \cite[Theorem~4.3]{Z22} shows that the $g_\alpha$-balayage $\mu^F_{g_\alpha}$, uniquely determined within $\breve{\mathfrak M}^+(F)$ by means of (\ref{bal1'}), can actually be found as the orthogonal projection of $\mu$ onto the convex, strongly complete cone $\mathcal E^+_{g_\alpha}(F)$ in the pre-Hilbert space $\mathcal E_{g_\alpha}:=\mathcal E^+_{g_\alpha}-\mathcal E^+_{g_\alpha}$, equipped with the inner product $\langle\lambda,\theta\rangle_{g_\alpha}:=I_{g_\alpha}(\lambda,\theta):=\int g_\alpha(x,y)\,d(\lambda\otimes\theta)(x,y)$. Since, according to Edwards \cite[Theorem~1.12.3]{E2}, such an orthogonal projection is the unique solution to problem (\ref{pr}), this completes the whole proof.

\subsection{Proof of Theorem~\ref{th-char}}\label{th-char-proof} Clearly, $\mu^F_{g_\alpha}\in\breve{\Gamma}_{g_\alpha}(\mu,F;D)$, for $\mu^F_{g_\alpha}\in\breve{\mathfrak M}^+(F)$ and
\begin{equation}\label{ch1}
U^{\mu^F_{g_\alpha}}_{g_\alpha}=U_{g_\alpha}^\mu\quad\text{q.e.\ on $F$}
\end{equation}
(Theorem~\ref{bal-ex}). To verify (\ref{eqchar}), we thus only need to show that
\begin{equation}\label{e-ch}
U^{\mu^F_{g_\alpha}}_{g_\alpha}\leqslant U^\theta_{g_\alpha}\quad\text{on $D$},
\end{equation}
where $\theta$ is an arbitrary measure of the class $\mathfrak M^+(D)$ having the property
\begin{equation}\label{ineqG}
U_{g_\alpha}^\theta\geqslant U_{g_\alpha}^\mu\quad\text{q.e.\ on $F$}.
\end{equation}
But combining (\ref{ch1}) and (\ref{ineqG}) gives $U_{g_\alpha}^\theta\geqslant U_{g_\alpha}^{\mu^F_{g_\alpha}}$ q.e.\ on $F$, hence $\mu^F_{g_\alpha}$-a.e., $\mu^F_{g_\alpha}$ being $c_{\kappa_\alpha}$-absolutely continuous, and (\ref{e-ch}) follows at once by applying Theorem~\ref{th-dom}.

If now $\zeta\in\breve{\Gamma}_{g_\alpha}(\mu,F;D)$ is another measure meeting (\ref{eqchar}), then $U_{g_\alpha}^\zeta=U_{g_\alpha}^{\mu^F_{g_\alpha}}$ on all of $D$, hence $I_{g_\alpha}(\zeta-\mu^F_{g_\alpha})=0$, the measure $\zeta+\mu^F_{g_\alpha}$ being likewise $c_{\kappa_\alpha}$-absolutely continuous, and consequently $\zeta=\mu^F_{g_\alpha}$, the kernel $g_\alpha$ being strictly positive definite.

\subsection{Proof of Theorem~\ref{th-balM1}} By virtue of (\ref{bal2}) and (\ref{est2}), the latter with $D:=\mathbb R^n$,
\begin{equation*}\mu^F_{g_\alpha}(D)=\mu^{F\cup Y}_{\kappa_\alpha}(F)\leqslant\mu^{F\cup Y}_{\kappa_\alpha}(\mathbb R^n)\leqslant\mu(\mathbb R^n)=\mu(D).\end{equation*}
On account of $\mu|_{F^r}\in\breve{\mathfrak M}^+$ and $\mu\ne\mu|_{F^r}$, applying Theorem~\ref{th-balR} gives
\[\mu^{F\cup Y}_{\kappa_\alpha}=\mu|_{F^r}+\tau^{F\cup Y}_{\kappa_\alpha},\quad\text{where $\tau:=\mu|_{D\setminus F^r}\ne0$},\]
and hence (\ref{Assa}) is reduced to proving
\begin{equation}\label{nonzero}
\tau^{F\cup Y}_{\kappa_\alpha}(Y)>0,
\end{equation}
which however follows at once from the description of the support $S\bigl(\tau^{F\cup Y}_{\kappa_\alpha}\bigr)$. Indeed, according to \cite[Corollary~5.4]{Z-bal2}, $S\bigl(\tau^{F\cup Y}_{\kappa_\alpha}\bigr)$ coincides with the {\it reduced kernel} of $F\cup Y$, defined as the set of all $x\in F\cup Y$ such that, for any $r>0$, $B(x,r)\cap(F\cup Y)$ is of nonzero $\kappa_\alpha$-capacity. In view of the latter assumption in (\ref{As}), this yields
\[S\bigl(\tau^{F\cup Y}_{\kappa_\alpha}\bigr)\supset{\rm Int}_{\mathbb R^n}Y\ne\varnothing,\]
whence (\ref{nonzero}).

\subsection{Proof of Theorem~\ref{th-balM2}} Let $\mu$ be an arbitrary measure of the class $\mathfrak M^+(D)$ having the properties $\mu|_F\in\breve{\mathfrak M}^+$ and $\mu\ne\mu|_F$. We aim to show that
\begin{equation}\label{gmu}
\mu^F_{g_\alpha}(D)=\mu(D),
\end{equation}
or equivalently\footnote{The equivalence of (\ref{gmu}) and (\ref{gmu'}) is clear from Theorem~\ref{bal-ex}, for,  because of $\mu|_F\in\breve{\mathfrak M}^+$,
\[\mu^F_{g_\alpha}=(\mu|_F)^F_{g_\alpha}+(\mu|_\Omega)^F_{g_\alpha}=\mu|_F+(\mu|_\Omega)^F_{g_\alpha}.\]}
\begin{equation}\label{gmu'}
\tau^F_{g_\alpha}(D)=\tau(D),\quad\text{where $\tau:=\mu|_\Omega\ne0$},
\end{equation}
holds true if and only if so does either of (ii) and (iii).

To this end, we first note that, according to Theorem~\ref{bal-ex},
\begin{equation}\label{gmu1}
\tau^F_{g_\alpha}(D)=\tau^{\Omega^c}_{\kappa_\alpha}(F),
\end{equation}
while, by virtue of (\ref{est2}) with $D:=\mathbb R^n$,
\begin{equation}\label{gmu2}
\tau^{\Omega^c}_{\kappa_\alpha}(\mathbb R^n)=\tau^{\Omega^c}_{\kappa_\alpha}(F)+\tau^{\Omega^c}_{\kappa_\alpha}(Y)\leqslant\tau(\mathbb R^n)=\tau(D).
\end{equation}
It is a simple matter to observe from (\ref{gmu1}) and (\ref{gmu2}) that (\ref{gmu'}), or equivalently
\begin{equation*}
\tau^{\Omega^c}_{\kappa_\alpha}(F)=\tau(D)\quad\text{for any $\tau\in\mathfrak M^+(\Omega)$},
\end{equation*}
is fulfilled if and only if the two equalities take place:
\begin{align}\label{al1}
\tau^{\Omega^c}_{\kappa_\alpha}(\mathbb R^n)&=\tau(\mathbb R^n),\\
\tau^{\Omega^c}_{\kappa_\alpha}(Y)&=0.\label{al2}
\end{align}
By use of the integral representation for $\kappa_\alpha$-swept measures \cite[Theorem~5.1]{Z-bal2}, we next see that (\ref{al1}) and (\ref{al2}) are equivalent, in turn, to the following two:
\begin{align}\label{al11}
(\varepsilon_x)^{\Omega^c}_{\kappa_\alpha}(\mathbb R^n)&=1\quad\text{for all $x\in\Omega$},\\
(\varepsilon_x)^{\Omega^c}_{\kappa_\alpha}(Y)&=0\quad\text{for all $x\in\Omega$}.\label{al21}
\end{align}

But, on account of (\ref{def-h}) and (\ref{Def}), (\ref{al11}) and (\ref{al21}) mean that, for all $x\in\Omega$,
\begin{equation*}
\omega_\alpha(x,\{\infty_{\mathbb R^n}\};\Omega)=0\quad\text{and}\quad
\omega_\alpha(x,Y;\Omega)=0,
\end{equation*}
whence (i)$\iff$(ii). Again in view of (\ref{Def}), we infer from \cite[Theorem~3.22]{FZ} that $\omega_\alpha(x,\{\infty_{\mathbb R^n}\};\Omega)$ equals $0$ on all of $\Omega$ if and only if $\Omega^c$ is not $\alpha$-thin at infinity. Together with what was shown just above, this proves (ii)$\iff$(iii), whence the theorem.

\subsection{Proof of Theorem~\ref{th-eq}}\label{th-eq-proof} Assume $c_{\kappa_\alpha}(Y)>0$, for if not, (\ref{eq-eq})--(\ref{eq-eq3}) are  equivalent by virtue of Corollary~\ref{cor-balM0}, while, according to \cite[Theorem~2.1]{Z-bal2}, (\ref{eq-eq3}) is necessary and sufficient for the existence of the $\kappa_\alpha$-equilibrium measure $\gamma_{\Omega^c,\kappa_\alpha}$, uniquely determined within $\breve{\mathfrak M}^+(\Omega^c)$ by means of the equality $U_{\kappa_\alpha}^{\gamma_{\Omega^c,\kappa_\alpha}}=1$ q.e.\ on $\Omega^c$. Noting from $c_{\kappa_\alpha}(Y)=0$ that $\breve{\mathfrak M}^+(\Omega^c)=\breve{\mathfrak M}^+(F)$ as well as that $U_{g_\alpha}^{\gamma_{\Omega^c,\kappa_\alpha}}=U_{\kappa_\alpha}^{\gamma_{\Omega^c,\kappa_\alpha}}$ (cf.\ Lemma~\ref{l-FZ}), we see that the same $\gamma_{\Omega^c,\kappa_\alpha}$ also serves as the $g_\alpha$-equilibrium measure $\gamma_F=\gamma_{F,g_\alpha}$.

It thus remains to verify the former part of the theorem. Assuming, along with $c_{\kappa_\alpha}(Y)>0$, that $\gamma_F$ exists, we first establish the following Lemmas~\ref{lemma1} and \ref{lemma2}.

\begin{lemma}\label{lemma1} The $g_\alpha$-equilibrium potential $U^{\gamma_F}_{g_\alpha}$ has the properties
\begin{align}\label{E1}
U^{\gamma_F}_{g_\alpha}&=1\quad\text{on $F^r$},\\
U^{\gamma_F}_{g_\alpha}&\leqslant1\quad\text{on $D$},\label{E2}\\
U^{\gamma_F}_{g_\alpha}&<1\quad\text{on $\Omega$}.\label{E3}
\end{align}
\end{lemma}

\begin{proof} Consider an increasing sequence $(K_j)$ of compact sets with the union $F$; then, by virtue of \cite[Theorem~4.12]{FZ},
\begin{equation}\label{kj}
U_{g_\alpha}^{\gamma_{K_j}}=1\quad\text{on $K_j^r$}.
\end{equation}
Given $z\in F^r$, choose $K_j$ so that $z\in K_j^r$. Since $\gamma_{K_j}=(\gamma_F)^{K_j}_{g_\alpha}$, cf.\ (\ref{EQ'}), applying (\ref{bal1'}) and (\ref{kj}) therefore implies that $U_{g_\alpha}^{\gamma_F}(z)=U_{g_\alpha}^{\gamma_{K_j}}(z)=1$, whence (\ref{E1}).

The measure $\gamma_F$ being $c_{\kappa_\alpha}$-absolutely continuous, we thus have $U^{\gamma_F}_{g_\alpha}=1$ $\gamma_F$-a.e., or equivalently, by Lemma~\ref{l-FZ},
$U^{\gamma_F}_{\kappa_\alpha}=1+U^{(\gamma_F)^Y_{\kappa_\alpha}}_{\kappa_\alpha}$ $\gamma_F$-a.e., which results in (\ref{E2}) by utilizing \cite[Theorem~1.V.10]{Doob} if $\alpha=2$, or \cite[Theorem~1.29]{L} otherwise.

To verify (\ref{E3}), assume to the contrary that there is $y\in\Omega$ with $U^{\gamma_F}_{g_\alpha}(y)=1$, or equivalently   $U^{\gamma_F}_{\kappa_\alpha}(y)=1+U^{(\gamma_F)^Y_{\kappa_\alpha}}_{\kappa_\alpha}(y)$.
As both $U^{\gamma_F}_{\kappa_\alpha}$ and $U^{(\gamma_F)^Y_{\kappa_\alpha}}_{\kappa_\alpha}$ are $\alpha$-sup\-er\-har\-monic on $\mathbb R^n$, while they are continuous and $\alpha$-harmonic on $\Omega$, this implies, by employing (\ref{E1}), (\ref{E2}), and \cite[Theorems~1.1, 1.28]{L}, that
\begin{equation}\label{kjj}U^{\gamma_F}_{\kappa_\alpha}=1+U^{(\gamma_F)^Y_{\kappa_\alpha}}_{\kappa_\alpha}\quad\text{q.e.\ on $D$}.\end{equation}
Since $c_{\kappa_\alpha}(Y)>0$, Lusin's type theorem \cite[Theorem~3.6]{L} applied to each of $U^{\gamma_F}_{\kappa_\alpha}$ and $U^{(\gamma_F)^Y_{\kappa_\alpha}}_{\kappa_\alpha}$ shows that (\ref{kjj}) contradicts the equality $U^{\gamma_F}_{\kappa_\alpha}=U^{(\gamma_F)^Y_{\kappa_\alpha}}_{\kappa_\alpha}$ q.e.\ on $Y$.
This proves (\ref{E3}), whence the lemma.
\end{proof}

\begin{lemma}\label{lemma2}Let $\mu$ be as required in Theorem~\ref{bal-ex}. Then
\begin{equation}\label{TM}
\mu^F_{g_\alpha}(F)=\int U^{\gamma_F}_{g_\alpha}\,d\mu.
\end{equation}
\end{lemma}

\begin{proof}
As $U^{\gamma_F}_{g_\alpha}=1$ q.e.\ on $F$ (Definition~\ref{def-eq}), hence $\mu^F_{g_\alpha}$-a.e., $\mu^F_{g_\alpha}$ being $c_{\kappa_\alpha}$-ab\-s\-ol\-ut\-e\-ly continuous (Theorem~\ref{bal-ex}), while $U_{g_\alpha}^{\mu^F_{g_\alpha}}=U_{g_\alpha}^\mu$ q.e.\ on $F$ (Theorem~\ref{bal-ex}), hence $\gamma_F$-a.e., $\gamma_F$ being likewise $c_{\kappa_\alpha}$-ab\-s\-ol\-utely continuous, the Lebesgue--Fubini theorem gives
\[
 \mu^F_{g_\alpha}(F)=\int U^{\gamma_F}_{g_\alpha}\,d\mu^F_{g_\alpha}=\int U_{g_\alpha}^{\mu^F_{g_\alpha}}\,d\gamma_F=\int U_{g_\alpha}^\mu\,d\gamma_F=\int U^{\gamma_F}_{g_\alpha}\,d\mu
\]
as claimed.\end{proof}

Now, fix $\mu\in\mathfrak M^+(\Omega)$, $\mu\ne0$. Applying (\ref{E3}) and (\ref{TM}) gives (\ref{eq-eq}), whence (\ref{eq-eq1}) according to Theorem~\ref{th-balM2}. This completes the proof of the theorem.

\subsection{Proof of Theorem~\ref{th-D}}\label{th-D-proof} Assume that (\ref{zeroo}) holds, and fix $\mu,\nu\in\mathfrak M^+(D)$ meeting (\ref{insigma}) and such that $\mu|_F\in\breve{\mathfrak M}^+(F)$ . Then, by virtue of Theorem~\ref{th-balM2}(i),
\begin{equation}\label{mua}
\mu^F_{g_\alpha}(D)=\mu(D).
\end{equation}
Noting from (\ref{insigma}) that $\nu\in\Gamma_{g_\alpha}(\mu,F;D)$, where $\Gamma_{g_\alpha}(\mu,F;D)$ was introduced by (\ref{gamma}), we conclude from Theorem~\ref{th-char} that $U_{g_\alpha}^{\mu^F_{g_\alpha}}\leqslant U_{g_\alpha}^\nu$ on all of $D$,
and hence, by applying \cite[Theorem~4.13]{FZ} (cf.\ Section~\ref{sec-Deny} above),
\[\mu_{g_\alpha}^F(D)\leqslant\nu(D).\]
Combining this with (\ref{mua}) gives $\mu(D)\leqslant\nu(D)$, as was to be proved.

\subsection{Proof of Theorem~\ref{th-D'}}\label{th-D'-proof} Assume (\ref{zeroo}) fails. Then, by virtue of Theorem~\ref{th-balM2}, one can choose $\mu_0\in\mathfrak M^+(\Omega)$ so that
\begin{equation}\label{mub}(\mu_0)^F_{g_\alpha}(D)<\mu_0(D).\end{equation}
But according to (\ref{bal1'}) applied to $\mu_0$, $U_{g_\alpha}^{(\mu_0)^F_{g_\alpha}}=U_{g_\alpha}^{\mu_0}$ q.e.\ on $F$,
which together with (\ref{mub}) validates the theorem for the above $\mu_0$ and $\nu_0:=(\mu_0)^F_{g_\alpha}$.

\section{Acknowledgements} This research was supported in part by a grant from the Simons Foundation (1030291, N.V.Z.). The author thanks Krzysztof Bogdan and Stephen J.\ Gardiner for helpful discussions on the content of the paper, and Douglas~P.\ Harding for reading and commenting on the manuscript.

\end{document}